%% file: Excluding_long_paths.tex
\documentclass[12pt]{amsart}
\usepackage{stmaryrd}
\usepackage{pifont}
\usepackage{amsmath, mathtools}
\usepackage{mathrsfs}
\usepackage{bbm}
\usepackage{amsfonts}
\usepackage{latexsym,amssymb,amsmath,amscd,amsfonts,epsfig,graphicx}

\usepackage{amsthm}

\usepackage[margin=1.4in]{geometry}

\newtheorem{thm}{Theorem}[section]
\newtheorem{lem}[thm]{Lemma}

\newtheorem{prop}[thm]{Property}%[section]
\newtheorem{obs}[thm]{Observation}%[section]

\newtheorem{oper}[thm]{Operation}

\def\pf{\noindent{\it Proof.\;\;\:}}
\def\qed{\nopagebreak\hfill{\rule{4pt}{7pt}}
\medbreak}

\DeclareMathOperator{\diam}{diam}

\DeclareMathOperator{\tdi}{tdi}
\DeclareMathOperator{\tw}{tw}

\theoremstyle{definition}
\title[Excluding long paths]{Excluding long paths}

\author[Bin Jia]{Bin Jia}

\thanks{Bin Jia gratefully acknowledges scholarships provided by The University of Melbourne.}

\begin{document}

\maketitle

\begin{center}
Department of Mathematics and Statistics\\
The University of Melbourne\\
Victoria 3010, Australia.\\
Email: jiabinqq@gmail.com \quad Mobile: +61 404639816
\end{center}

\medskip

\input{Abstract}

\input{Introduction}

\input{Terminology}

\input{BetterQuasiOrderings}

\bibliographystyle{plain}
\bibliography{Xbib}

\end{document}

%% file: Abstract.tex
\noindent {\bf Abstract.}
Ding (1992) proved that for each integer ${m} \geqslant 0$, and every infinite sequence of finite simple graphs $G_1, G_2, \ldots$, if none of these graphs contains a path of length ${m}$ as a subgraph, then there are indices $i < j$ such that $G_i$ is isomorphic to an induced subgraph of $G_j$. We generalise this result to infinite graphs, possibly with parallel edges and loops.

%Further, we show that every sequence of ${m}$-path roots of a finite graph is better-quasi-ordered by the subgraph relation, and by the induced subgraph relation if these ${m}$-path roots have bounded multiplicity.\\

\medskip


%% file: Introduction.tex
\section{Introduction and main results}
All graphs in this paper are undirected. Unless stated otherwise, a graph may be finite or infinite, and may contain parallel edges and loops. Let $m \geqslant 0$ be an integer. We use $P_m$ to denote a path with $m$ edges (and $m + 1$ vertices).

Robertson and Seymour \cite{GMXX} proved that the finite graphs are well-quasi-ordered by the minor relation. Thomas \cite{Thomas1988cexample} found an example showing that the infinite graphs are not well-quasi-ordered by the minor relation. Later, Thomas~\cite{ThomasRobin1989} proved that the finite or infinite graphs without a given finite planar graph as a minor are well-quasi-ordered (furthermore, better-quasi-ordered) by the minor relation.

A \emph{Robertson chain} of length $m$ is the graph obtained by duplicating each edge of $P_m$. Robertson conjectured in 1980's that the finite graphs without a Robertson chain of length $m$ as a topological minor are well-quasi-ordered by the topological minor relation. This conjecture was proved by Liu \cite{Liu2014}.

By considering the \emph{type} of a finite simple graph, Ding \cite{Ding1992} proved that, for each integer $m \geqslant 0$, the finite simple graphs without $P_m$ as a subgraph are well-quasi-ordered by the induced subgraph relation. Another proof, based on the \emph{tree-depth}, was given by Ne{\v{s}}et{\v{r}}il and Ossona de Mendez \cite{NesetrilOssona2012}. We generalise Ding's theorem to infinite graphs, possibly with parallel edges and loops.

\begin{thm}\label{thm_lpathfreebqo}
Given a finite graph $H$, the graphs (respectively, of bounded multiplicity) without $H$ as a subgraph are better-quasi-ordered by the (respectively, induced) subgraph relation if and only if $H$ is a disjoint union of paths.
\end{thm}

Let $t \geqslant 1$ be an integer. A \emph{$t$-dipole} is a graph with two vertices and $t$ edges between them. Clearly, a $t$-dipole does not contain $P_2$ as a subgraph. Further, the set of $t$-dipoles, for $t = 1, 2, \ldots$, is well-quasi-ordered by the subgraph relation, but not by the induced subgraph relation.

Our method in dealing with the graphs without $P_m$ as a subgraph is different from the methods of Ding \cite{Ding1992} and Ne{\v{s}}et{\v{r}}il and Ossona de Mendez \cite{NesetrilOssona2012}. Instead of studying the type or the tree-depth of a graph $G$, we prove Theorem \ref{thm_lpathfreebqo} by investigating the tree-decompositions of $G$. A key step is to show that, if $G$ does not contain $P_m$ as a subgraph, then $G$ admits a tree-decomposition which attains the minimum width such that the diameter of the tree for the tree-decomposition is bounded by a function of $m$. Then we prove Theorem \ref{thm_lpathfreebqo} by induction on the diameter.

%% file: Terminology.tex
\section{Terminology}\label{sec_term}
This section presents some necessary definitions and basic results about tree-decompositions and quasi-orderings of graphs.

%
%Let $G$ be a graph, and $u, v \in V(G)$. Denote by $E_G(u, v)$ the set of edges of $G$. In particular, if $u$ and $v$ are identical, then $E_G(u, v)$ is either empty, or consists of loops. The \emph{multiplicity} of $G$ is the superior of $|E_G(u, v)|$ over $u, v \in V(G)$.

A binary relation on a set is a \emph{quasi-ordering} if it is reflexive and transitive. A quasi-ordering $\leqslant$ on a set $\mathbf{Q}$ is a \emph{well-quasi-ordering} if for every infinite sequence $q_1, q_2, \ldots$ of $\mathbf{Q}$, there are indices $i < j$ such that $q_i \leqslant q_j$. And if this is the case, then $q_i$ and $q_j$ are called a \emph{good pair}, and $\mathbf{Q}$ is \emph{well-quasi-ordered} by $\leqslant$.

%The notions of tree-decomposition and tree-width were studied extensively by Robertson and Seymour in proving that finite graphs are well-quasi-ordered by the minor relation \cite{GMXX}.

Let $G$ be a hypergraph, $T$ be a tree, and $\mathcal{V} := \{V_v \mid v \in V(T)\}$ be a set cover of $V(G)$ indexed by $v \in V(T)$. The pair $(T, \mathcal{V})$ is called a \emph{tree-decomposition} of $G$ if the following two conditions are satisfied:
\begin{itemize}
\item for each hyperedge $e$ of $G$, there exists some $V \in \mathcal{V}$ containing all the vertices of $G$ incident to $e$;
\item for every path $[v_0, \ldots, v_i, \ldots, v_m]$ of $T$, we have $V_{v_0} \cap V_{v_{m}} \subseteq V_{v_i}$.
\end{itemize}

The \emph{width $\tw(T, \mathcal{V})$} of $(T, \mathcal{V})$ is \emph{$\sup\{|V| - 1\mid V \in \mathcal{V}\}$}. The \emph{tree-width $\tw(G)$} of $G$ is the minimum width of a tree-decomposition of $G$. The \emph{tree-diameter $\tdi(G)$} of $G$ is the minimum diameter of $T$ over the tree-decompositions $(T, \mathcal{V})$ of $G$ such that $\tw(T, \mathcal{V}) = \tw(G)$.

%for each edge $e$ of $G$, there exists some $V \in \mathcal{V}$ containing all vertices incident to $e$, and for every path $[w_0, \ldots, w_i, \ldots, w_{m}]$ of $T$, we have $V_{w_0} \cap V_{w_{m}} \subseteq V_{w_i}$.

%Let $T$ be a tree, and $\mathcal{V} := \{V_t | t \in V(T)\}$ be a set cover of $V(G)$ indexed by the nodes of $T$. The pair $(T, \mathcal{V})$ is called a \emph{tree-decomposition} of $G$ if, first of all, every edge of $G$ is in some $V \in \mathcal{V}$. And secondly, for every path $[t_0, \ldots, t_{m}]$ of $T$ and $0 \leqslant i \leqslant {m}$, $V_{t_0} \cap V_{t_{m}}$ is included in $V_{t_i}$. The \emph{width} of $(T, \mathcal{V})$ is defined to be $\sup\{|V| - 1| V \in \mathcal{V}\}$. The \emph{tree-width $\tw(G)$} of $G$ is the minimum width over all the tree-decompositions of $G$. The \emph{tree-diameter $\tdi(G)$} of $G$ is the minimum $\diam(T)$ over all the tree-decompositions $(T, \mathcal{V})$ of $G$ of width $\tw(G)$.

Let $[v_0, e_1, v_1, \ldots, v_{m-1}, e_m, v_{m}]$ be a path of $T$. Denote by \emph{$\mathcal{V}_T(v_0, v_{m})$} the set of minimal sets, up to the subset relation, among $V_{v_i}$ and $V_{e_j} := V_{v_{j - 1}} \cap V_{v_j}$ for $i \in \{0, 1, \ldots, {m}\}$ and $j \in [{m}] := \{1, 2, \ldots, {m}\}$. We identify repeated sets in $\mathcal{V}_T(v_0, v_{m})$.

A tree-decomposition $(T, \mathcal{V})$ of $G$ is said to be \emph{linked} if
\begin{itemize}
\item for every pair of nodes $u$ and $v$ of $T$, and subsets $U$ of $V_u$ and $V$ of $V_v$ such that $|U| = |V| =: k$, either $G$ contains $k$ disjoint paths from $U$ to $V$, or there exists some $W \in \mathcal{V}_T(u, v)$ such that $|W| < k$.
\end{itemize}

Kruskal's theorem \cite{Kruskal1960} states that finite trees are well-quasi-ordered by the topological minor relation. Nash-Williams \cite{NashWilliams1965} generalised this theorem and proved that infinite trees are \emph{better-quasi-ordered} by the same relation. Let $\mathcal{A}$ be the set of all finite ascending sequences of nonnegative integers. For $A, B \in \mathcal{A}$, write $A <_{\mathcal{A}} B$ if $A$ is a strict initial subsequence of some $C \in \mathcal{A}$, and by deleting the first term of $C$, we obtain $B$. Let $\mathcal{B}$ be an infinite subset of $\mathcal{A}$, and $\bigcup\mathcal{B}$ be the set of nonnegative integers appearing in some sequence of $\mathcal{B}$. $\mathcal{B}$ is called a \emph{block} if it contains an initial subsequence of every infinite increasing sequence of $\bigcup\mathcal{B}$. Let $\mathbf{Q}$ be a set with a quasi-ordering $\leqslant_{\mathbf{Q}}$. A \emph{$\mathbf{Q}$-pattern} is a function from a block $\mathcal{B}$ into $\mathbf{Q}$. A $\mathbf{Q}$-pattern $\varphi$ is \emph{good} if there exist $A, B \in \mathcal{B} \subseteq \mathcal{A}$ such that $A <_{\mathcal{A}} B$ and $\varphi(A) \leqslant_{\mathbf{Q}} \varphi(B)$. $\mathbf{Q}$ is said to be \emph{better-quasi-ordered} by $\leqslant_{\mathbf{Q}}$ if every $\mathbf{Q}$-pattern is good. For example, the set of nonnegative integers is better-quasi-ordered by the natural ordering. It follows from the definitions that a better-quasi-ordering is a well-quasi-ordering. And $\mathbf{Q}$ is better-quasi-ordered if and only if each subset of $\mathbf{Q}$ is better-quasi-ordered.

For an integer $j \geqslant 1$, define a quasi-ordering on $\mathbf{Q}^j$ as follows: $(q_1, \ldots, q_j) \leqslant_{\mathbf{Q}^j} (q'_1, \ldots, q'_j)$ if $q_i \leqslant_{\mathbf{Q}} q'_i$ for every $i \in [j]$. The following lemma follows from the Galvin-Prikry theorem  \cite{GalvinPriskry1973} (see also \cite[(3.11)]{ThomasRobin1989} and \cite[Lemma 3]{KuhnDaniela2001}).

%A proof of the lemma below is given by K{\"u}hn \cite{KuhnDaniela2001}.

%directly from the Galvin-Prikry theorem  \cite{GalvinPriskry1973}.

\begin{lem}\label{lem_coverbqo}
Let $k \geqslant 1$ be an integer, and $\mathbf{Q} = \bigcup_{i = 1}^k\mathbf{Q}_i$ be a quasi-ordered set. Then the following statements are equivalent:
\begin{itemize}
\item[{\bf (1)}] $\mathbf{Q}$ is better-quasi-ordered.
\item[{\bf (2)}] $\mathbf{Q}_i$ is better-quasi-ordered for every $i \in [k]$.
\item[{\bf (3)}] $\mathbf{Q}^j$ is better-quasi-ordered for every integer $j \geqslant 1$.
\end{itemize}
\end{lem}

%$\mathbf{Q} = \bigcup_{i = 1}^k\mathbf{Q}_i$ is better-quasi-ordered if and only if for every $i \in [k]$, $\mathbf{Q}_i$ is better-quasi-ordered if and only if for every $j \geqslant 1$, $\mathbf{Q}^j$ is better-quasi-ordered.

Define a quasi-ordering $\leqslant$ on the powerset of $\mathbf{Q}$ as follows. For $S_1, S_2 \subseteq \mathbf{Q}$, write $S_1 \leqslant S_2$ if there is an injection $\varphi$ from $S_1$ to $S_2$ such that $q \leqslant_{\mathbf{Q}} \varphi(q)$ for every $q \in S_1$.

Let $\mathbf{Q}$ be a set with a quasi-ordering $\leqslant_\mathbf{Q}$.
Let $\mathcal{S}$ be a set of sequences whose elements are from $\mathbf{Q}$. For $S_1 := (q_1, q_2, \ldots) \in \mathcal{S}$ and $S_2 \in \mathcal{S}$, we say $S_1 \leqslant_\mathcal{S} S_2$ if there is a subsequence $S_3 := (p_1, p_2, \ldots)$ of $S_2$ such that $S_1$ and $S_3$ have the same length, and that $q_i \leqslant_\mathbf{Q} p_i$ for every index $i$ used in $S_1$. The following results are due to Nash-Williams \cite{NashWilliams1968}.

\begin{lem}[\cite{NashWilliams1968}]\label{lem_subseqbqo}
Every finite quasi-ordered set is better-quasi-ordered. And each better-quasi-ordering is a well-quasi-ordering. Moreover, a quasi-ordered set
$\mathbf{Q}$ is better-quasi-ordered if and only if the powerset of $\mathbf{Q}$ is better-quasi-ordered if and only if every set of sequences whose elements are from $\mathbf{Q}$ is better-quasi-ordered.
\end{lem}

%% file: BetterQuasiOrderings.tex
\section{Graphs without $P_m$ as a subgraph}\label{sec_pathGraphs}
In this section, we show that a graph without a given finite path as a subgraph has bounded tree-diameter. We achieve this by modifying a given tree-decomposition $(T, \mathcal{V})$ of $G$ such that $\diam(T)$ is reduced but $\tw(T, \mathcal{V})$ remain unchanged.

One of the easiest ways to reduce $\diam(T)$ is to delete repeated sets in $\mathcal{V}$. Let $(T, \mathcal{V})$ be a linked tree-decomposition of $G$. Suppose there are different nodes $u, v$ of $T$ such that $V_u = V_v$. Let $\mathcal{V}_1$ be obtained from $\mathcal{V}$ by deleting the set $V_v \in \mathcal{V}$ indexed by $v \in V(T)$. Let $T_1$ be obtained from $T$ by contracting an edge between $v$ and some $w \in V(T)$ to a node $w$ of $T$. Then $(T_1, \mathcal{V}_1)$ is a linked-tree decomposition of $G$ such that $\diam(T_1) \leqslant \diam(T)$ and $\tw(T_1, \mathcal{V}_1) = \tw(T, \mathcal{V})$.

We emphasise that, to reduce $\diam(T)$, it is not enough to just remove repeated sets in $\mathcal{V}$. We also must deal with repeated sets in $\{V_e \mid e \in E(T)\}$. Let $m \geqslant 3$ be an integer, and $G$ be a star with center $0$ and leaves $1, 2, \ldots, m$. Let $T := [v_1, \ldots, v_m]$ be a path, $V_{v_i} := \{0, i\}$ for $i \in [m]$, and $\mathcal{V} := \{V_{v_i}\mid i \in [m]\}$. Then $(T, \mathcal{V})$ is a linked tree-decomposition of $G$ with width $1$ and without repeated sets in $\mathcal{V}$. However, $\diam(T) = m$ is too large for $G$. In fact, $\tdi(G) = 2$. To see this, let $T^*$ be a star with center $v_m$, and with leaves $v_1, \ldots, v_{m - 1}$. Then $(T^*, \mathcal{V})$ is a linked tree-decomposition of $G$ with width $1$ such that $\diam(T^*) = 2$. Note that $\mathcal{V}$ is not changed. And $T^*$ can be obtained from $T$ by deleting edges between $v_i$ and $v_{i + 1}$, and adding an extra edge between $v_m$ and $v_i$ for $i \in [m - 1]$.

The operation above can be extended to deal with a tree-decomposition $(T, \mathcal{V})$ such that $\{V_e \mid e \in E(T)\}$ contains repeated sets. To do this, we need the following lemma.

\begin{lem}\label{lem_v01vl1l}
Let ${m} \geqslant 1$ be an integer, $(T, \mathcal{V})$ be a tree-decomposition of a graph $G$, and $[v_0, \ldots, v_{m}]$ be a path of $T$. Let $U \subseteq V(G)$. Then $U = V_{v_0}\cap V_{v_1} = V_{v_{{m} - 1}}\cap V_{v_{m}}$ if and only if $U = V_{v_i} \cap V_{v_0} = V_{v_i} \cap V_{v_{m}}$ for all $i \in [{m} - 1]$.
\end{lem}
\pf $(\Leftarrow)$ Let $i = 1$, we have $U = V_{v_0}\cap V_{v_1}$. Let $i = m - 1$, we have $U = V_{v_{{m} - 1}}\cap V_{v_{m}}$. So $U = V_{v_0}\cap V_{v_1} = V_{v_{{m} - 1}}\cap V_{v_{m}}$.

$(\Rightarrow)$ Since $U = V_{v_0}\cap V_{v_1} = V_{v_{{m} - 1}}\cap V_{v_{m}}$, we have $U = V_{v_0}\cap V_{v_1} \cap V_{v_{{m} - 1}}\cap V_{v_{m}} \subseteq V_{v_0} \cap V_{v_{m}}$. Since $1 \leqslant i \leqslant m - 1$, by the definition of a tree-decomposition, $V_{v_0} \cap V_{v_{m}} \subseteq V_{v_i}$, and $V_{v_0} \cap V_{v_i} \subseteq V_{v_1}$. So $V_{v_0} \cap V_{v_{m}} \subseteq V_{v_0} \cap V_{v_i} \subseteq V_{v_0} \cap V_{v_1} = U$. Thus $U = V_{v_0}\cap V_{v_1} = V_{v_0} \cap V_{v_i} = V_{v_0} \cap V_{v_{m}}$. Symmetrically, $U = V_{v_i} \cap V_{v_{m}}$, and hence $U = V_{v_i} \cap V_{v_0} = V_{v_i} \cap V_{v_{m}}$.
\qed

We list the operation that can be used to reduce $\diam(T)$ for $(T, \mathcal{V})$.
\begin{oper}\label{oper_reduceDiam}
Let $(T, \mathcal{V})$ be a tree-decomposition of a finite graph $G$ such that $T$ is a finite tree. Let $U \subseteq V(G)$ such that $E_U := \{e \in E(T) \mid V_e = U\}$ is not empty. Let $T_U$ be the minimal subtree of $T$ containing $E_U$, and $u$ be a center of $T_U$. For each $e \in E_U$ with end vertices $v, w \in V(T) \setminus \{u\}$ such that $u$ is closer to $v$ than to $w$, delete $e$ and add an extra edge between $w$ and $u$.
\end{oper}

Let $T'$ be obtained from $T$ by applying Operation \ref{oper_reduceDiam} to a subset $U$ of $V(G)$. Let ${E'}_U := E(T')\setminus E(T)$. In Figure \ref{F:ExLoP-reduceDiameterT}, $\diam(T) = 6$, and the bold edges represent the edges of $E_U$. During the operation, the bold edge incident to $u$ does not change. Other bold edges are deleted. The curve edges in $T'$ represent the edges in ${E'}_U$. Note that $\diam(T') = 4$, less than the diameter of $T$.

\begin{figure}[!h]
\centering
\includegraphics{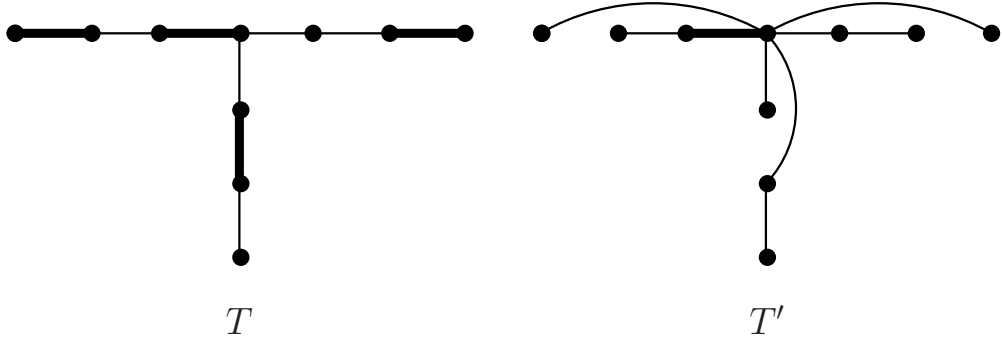}
\caption[Reducing the diameter of the tree for a tree-decomposition]{Reducing the diameter of the tree for a tree-decomposition}
\label{F:ExLoP-reduceDiameterT}
\end{figure}

Our next lemma is useful in proving that $(T', \mathcal{V})$ is a tree-decomposition of $G$.
\begin{lem}\label{lem_treeMiniParts}
Let $(T, \mathcal{V})$ be a tree-decomposition of a finite graph $G$ such that $T$ is a finite tree. Then $T'$ is a finite tree. And  for every pair of $x, y \in V(T) = V(T')$, we have $\mathcal{V}_T(x, y) = \mathcal{V}_{T'}(x, y)$.
\end{lem}
\pf If $E_U = \emptyset$ or $\diam(T_U) \leqslant 2$, then $T' = T$, and the lemma follows trivially. Now assume that $\diam(T_U) \geqslant 3$.
By Operation \ref{oper_reduceDiam}, we have that $T'$ is connected, $V(T') = V(T)$ and $|E(T')| = |E(T)|$. So $T'$ is a finite tree.

Let $P$ and $Q$ be paths from $x$ to $y$ in $T$ and $T'$ respectively. By Operation \ref{oper_reduceDiam}, we have that $E(P) \cap E_U = \emptyset$ if and only if $E(Q) \cap {E'}_U = \emptyset$. And if this is the case, then $P = Q$ and there is nothing to show. Now suppose that $E(P) \cap E_U \neq \emptyset$.

For each $f \in E(P) \cup E(Q)$, there are three cases: First, $f \in E_U \cup {E'}_U$ and $V_f = U$. Second, $V_f \neq U$, and there are two edges $e, e' \in E_U$ such that $f$ is on the path of $T$ between $e$ and $e'$. In this case, since $(\mathcal{V}, T)$ is a tree-decomposition of $G$, we have that $U = V_{e} \cap V_{e'} \subset V_f$.

Last, $f \notin E_U \cup {E'}_U$ and $f$ is not between two edges of $E_U$ in $T$. In this situation, assume for a contradiction that $f$ is in a cycle of $T \cup T'$. By Operation \ref{oper_reduceDiam}, $f$ is between $u$ and an edge $e \in E_U$ in $T$. Note that $u$ is a center of $T_U$. So there is another edge $e' \in E_U$ such that $u$ is on the path of $T$ from $e$ to $e'$, a contradiction. Thus $f$ is a bridge of $T \cup T'$. So $f \in E(P) \cap E(Q)$.

By the analysis above, both $\mathcal{V}_T(x, y)$ and $\mathcal{V}_{T'}(x, y)$ are $$\min\{U, V_f \mid f \in E(P) \cap E(Q)\}.$$
\qed

A tree-decomposition $(T, \mathcal{V})$ is \emph{short} if for every pair of different $e, f \in E(T)$, if $V_e = V_f$, then $e$ and $f$ are incident in $T$.
Let $T^*$ be obtained from $T$ by applying Operation \ref{oper_reduceDiam} to each $U \subseteq V(G)$.
In the following, we verify that $(T^*, \mathcal{V})$ is a short tree-decomposition.
\begin{lem}\label{lem_reducelength}
Let $(T, \mathcal{V})$ be a tree-decomposition of a finite graph $G$ such that $T$ is a finite tree. Then all the following statements hold:
\begin{itemize}
\item[{\bf (1)}] $T^*$ is a finite tree such that $\diam(T^*) \leqslant \diam(T)$.
\item[{\bf (2)}] $(T^*, \mathcal{V})$ is a tree-decomposition of $G$ such that $\tw(T^*, \mathcal{V}) = \tw(T, \mathcal{V})$.
\item[{\bf (3)}] $(T^*, \mathcal{V})$ is linked if and only if $(T, \mathcal{V})$ is linked.
\item[{\bf (4)}] For $e, f \in E(T^*)$, if $V_e = V_f$, then $e$ and $f$ are incident in $T^*$.
\end{itemize}
\end{lem}
\begin{proof}
{\bf (1)} follows from Operation \ref{oper_reduceDiam} and Lemma \ref{lem_treeMiniParts}.

For {\bf (2)}, let $x, y \in V(T^*) = V(T)$. Since $(T, \mathcal{V})$ is a tree-decomposition, $V_x \cap V_y$ is a subset of every set in $\mathcal{V}_T(x, y)$. By Lemma \ref{lem_treeMiniParts}, $V_x \cap V_y$ is a subset of every set in $\mathcal{V}_{T^*}(x, y)$. Let $z \in V(T)$ be on the path of $T$ from $x$ to $y$. By the definition of $\mathcal{V}_{T^*}(x, y)$, there exists some $V \in \mathcal{V}_{T^*}(x, y)$ such that $V \subseteq V_z$. So $V_x \cap V_y \subseteq V_z$ and hence $(T^*, \mathcal{V})$ is a tree-decomposition of $G$. Operation \ref{oper_reduceDiam} does not change a set in $\mathcal{V}$, so $\tw(T^*, \mathcal{V}) = \tw(T, \mathcal{V})$.

By Lemma \ref{lem_treeMiniParts}, for every pair of $x, y \in V(T)$, we have $\mathcal{V}_T(x, y) = \mathcal{V}_{T^*}(x, y)$. So
{\bf (3)} follows from the definition of a linked tree-decomposition.

{\bf (4)} follows from Operation \ref{oper_reduceDiam}.
\end{proof}

Introduced by K{\v{r}}{\'{\i}}{\v{z}} and Thomas \cite{KrizThomas1991}, an \emph{$M$-closure} of a simple graph $G$ is a triple $(T, \mathcal{V}, X)$, where $X$ is a chordal graph without a complete subgraph of order $\tw(G) + 2$, $V(G) = V(X)$, $E(G) \subseteq E(X)$, and $(T, \mathcal{V})$ is a linked tree-decomposition of $X$ such that each part induces a maximal complete subgraph of $X$. An $M$-closure is \emph{short} if the tree-decomposition is short.

%, we obtain a tree decomposition of $G$ as follows:
\begin{lem}\label{lem_linkDshort}
Every graph $G$ of finite tree-width, with or without loops, admits a short linked tree-decomposition of width $\tw(G)$.
\end{lem}
\pf It is enough to consider the case that $G$ is a simple graph. By K{\v{r}}{\'{\i}}{\v{z}} and Thomas \cite[(2.3)]{KrizThomas1991}, every finite simple graph has an $M$-closure $(T, \mathcal{V}, X)$. By Lemma \ref{lem_reducelength}, $(T^*, \mathcal{V}, X)$ is a short $M$-closure. In \cite[(2.4)]{KrizThomas1991}, replacing `an $M$-closure' with `a short $M$-closure' causes no conflict. So $G$ has a short $M$-closure. The rest of the lemma follows from a discussion similar with \cite[(2.2)]{KrizThomas1991}.
\qed

%Repeat the operations in Lemma \ref{lem_reducelength} on $(T, \mathcal{V})$ until no path of $T$ of length at least $3$ satisfies the conditions in Lemma \ref{lem_v01vl1l}. Delete the repeated parts, we obtain a short $M$-closure. In \cite[(2.4)]{KrizThomas1991}, replacing `an $M$-closure' with `a short $M$-closure' causes no conflict. So every graph has a short $M$-closure. The rest of the lemma follows from a discussion similar with \cite[(2.2)]{KrizThomas1991}.

%each path $[v_0, \ldots, v_{m}]$ of $T$, where ${m} \geqslant 3$, $V_{v_0} \cap V_{v_1} \neq V_{v_{{m} - 1}} \cap V_{v_{{m}}}$.
%The \emph{diameter $\tdi(T, \mathcal{V})$} of a tree-decomposition $(T, \mathcal{V})$ is the diameter of $T$.
%The \emph{tree-diameter $\tdi(G)$} of a graph $G$ is the minimum $\diam(T)$ over all tree-decompositions $(T, \mathcal{V})$ of $G$ such that $\tw(T, \mathcal{V}) = \tw(G)$.

Let $s \geqslant 0$ be an integer. For $i = 1, 2, \ldots$, let $T_i$ be a tree with a center $v_i$ and of diameter at most $s$. Let $T$ be obtained from these trees by adding an edge from $v_1$ to each of $v_2, v_3, \ldots$. Then $\diam(T) \leqslant 2\lceil\frac{s}{2}\rceil + 2 \leqslant s + 3$. Thus we have the following observation.

\begin{obs}\label{obs_tdiDiscConn}
Let $G$ be a graph, and $s$ be the maximum tree-diameter of a connected component of $G$. Then $\tdi(G) \leqslant s + 3$.
\end{obs}

We now show that graphs without a given path as a subgraph have bounded tree-diameter.

\begin{lem}\label{lem_tdibounded}
Let $G$ be a graph without $P_m$ as a subgraph. Then $G$ admits a linked tree-decomposition $(T, \mathcal{V})$ such that $\tw(T, \mathcal{V}) = \tw(G) \leqslant {m} - 1$, and $\diam(T) \leqslant 2({m}^2 - {m} + 2)^{m} + 1$. And if $G$ is connected, then $\diam(T) \leqslant 2({m}^2 - {m} + 2)^{m} - 2$.
\end{lem}
\begin{proof}
Let $X$ be a finite subgraph of $G$. Suppose for a contradiction that $\tw(X) \geqslant {m}$. Then by Robertson and Seymour \cite{RobertsonSeymour1983}, $X$ contains a path of length $m$, a contradiction. So $\tw(X) \leqslant {m} - 1$. By a compactness theorem for the notion of tree-width \cite{Thomas1988,Thomassen1989}, we have that $\tw(G) \leqslant {m} - 1$.

For the tree-diameter, by Observation \ref{obs_tdiDiscConn}, we only need to consider the case that $G$ is nonnull and connected. By Lemma \ref{lem_linkDshort}, $G$ admits a short linked tree-decomposition $(T, \mathcal{V})$ of width $\tw(G)$. Let $p := \tw(G) + 1 \in [m]$.

Let $P := [v_0, e_1, \ldots, e_s, v_s]$ be a path of length $s \geqslant 1$ in $T$. We say $P$ is \emph{$t$-rotund}, where $t \in [s]$, if there exists some $k \in [p]$ and a sequence $1 \leqslant i_1 < \ldots < i_t \leqslant s$ such that $V_{e_{i_1}}, \ldots, V_{e_{i_t}}$ are pairwise distinct, $|V_{e_{i_j}}| = k$ for all $j \in [t]$, and  $|V_{e_{j}}| \geqslant k$ for all $j$ such that $i_1 \leqslant j \leqslant i_t$. Let $s^* \in [s]$ be the maximum number of edges of $P$ corresponding to pairwise different subsets of $V(G)$.

\noindent{\bf Claim.} $s \leqslant 2s^*$. Since if $s \geqslant 2s^* + 1$, then there are $1 \leqslant j_1 < j_2 < j_3 \leqslant s$ such that $V_{e_{j_1}} = V_{e_{j_2}} = V_{e_{j_3}}$, contradicting the shortness of $(T, \mathcal{V})$.

\noindent{\bf Claim.} If $P$ is not $t$-rotund, then $s^* \leqslant t^p - 1$. To see this, let $s_k$ be the maximum number of edges of $P$ corresponding to pairwise different subsets with $k$ vertices of $V(G)$. Since $P$ is not $t$-rotund, we have that $s_1 \leqslant t - 1$. More generally, for $k \geqslant 2$, we have $s_k \leqslant (s_1 + \ldots + s_{k - 1} + 1)(t - 1)$. By induction on $k$ we have that $s_k \leqslant t^{k - 1}(t - 1)$ for each $k \in [p]$. So $s^* = s_1 + \ldots + s_p \leqslant t^p - 1$.

\noindent{\bf Claim.} If $P$ is $t$-rotund, then $t \leqslant p({m} - 1) + 1$. To prove this, recall that $(T, \mathcal{V})$ is a linked tree-decomposition. So there are $k$ disjoint paths in $G$ with at least $|\bigcup_{j = 1}^t V_{e_{i_j}}| \geqslant k + t - 1$ vertices. Since $G$ does not contain $P_m$ as a subgraph, each of these $k$ paths contains at most ${m}$ vertices. So $k + t - 1 \leqslant k{m}$. As a consequence, $t \leqslant k({m} - 1) + 1 \leqslant p({m} - 1) + 1$.

Now let $t$ be the maximum integer such that $P$ is $t$-rotund. By the third claim, $t \leqslant p({m} - 1) + 1$. Since $P$ is not $(t + 1)$-rotund, by the second claim, $s^* \leqslant (t + 1)^p - 1$. Thus $s \leqslant 2s^* \leqslant 2[(t + 1)^p - 1] \leqslant 2[p({m} - 1) + 2]^p - 2 \leqslant 2({m}^2 - {m} + 2)^{m} - 2$.
\end{proof}

\section{Better-quasi-ordering}
This section shows some better-quasi-ordering results for graphs without a given path as a subgraph.

A \emph{rooted hypergraph} is a hypergraph $G$ with a special designated subset $r(G)$ of $V(G)$. Note that $r(G)$ can be empty. Let $\mathbf{Q}$ be a set with a quasi-ordering $\leqslant_{\mathbf{Q}}$. A \emph{$\mathbf{Q}$-labeled rooted hypergraph} is a rooted hypergraph $G$ with a mapping $\sigma: E(G) \mapsto \mathbf{Q}$. The lemma below says that graphs with finitely many vertices (respectively, of bounded multiplicity) are better-quasi-ordered by the (respectively, induced) subgraph relation.
\begin{lem}\label{lem_nboundedbqo}
Let $\mathbf{Q}$ be a better-quasi-ordered set, and $\mathcal{G}$ be a sequence of $\mathbf{Q}$-labeled rooted hypergraphs (respectively, of bounded multiplicity) whose vertex sets are the subsets of $[p]$, where $p \geqslant 1$ is an integer. For $X, Y \in \mathcal{G}$, denote by $X \subseteq Y$ (respectively, $X \leqslant Y$) that $r(X) = r(Y)$, and there is an isomorphism $\varphi$ from $X$ to a (respectively, an induced) subgraph  of $Y$ such that for all $i \in V(X)$ and $e \in E(X)$, we have that $\varphi(i) = i$ and $\sigma(e) \leqslant_{\mathbf{Q}} \sigma(\varphi(e))$. Then $\mathcal{G}$ is better-quasi-ordered by $\subseteq$ (respectively, $\leqslant$).
\end{lem}
\pf
There are $\sum_{i = 0}^p\binom{p}{i}2^i = 3^p$ choices for vertex sets and roots. So by Lemma \ref{lem_coverbqo}, it is safe to assume that all $G \in \mathcal{G}$ have the same vertex set, say $[p]$, and the same root.

Then each $G$ can be seen as a sequence of length $2^p - 1$, indexed by the nonempty subsets of $[p]$. And for each nonempty $V \subseteq [p]$, the term of the sequence indexed by $V$ is the collection of elements of $\mathbf{Q}$ that are used to label the hyperedges $e$ of $G$ such that the set of end vertices of $e$ is $V$. By Lemmas \ref{lem_coverbqo} and \ref{lem_subseqbqo}, $\mathcal{G}$ is better-quasi-ordered by $\subseteq$.

Now let $\mu$ be an upper bound of the multiplicities. There are $(\mu + 1)^{2^p - 1}$ unequal hypergraphs of vertex set $[p]$. By Lemma \ref{lem_coverbqo}, we can assume that all these rooted hypergraphs are equal. In this situation, each $G \in \mathcal{G}$ is a sequence of length $2^p - 1$, indexed by the nonempty subsets of $[p]$. And for each nonempty $V \subseteq [p]$, the term of the sequence indexed by $V$ is the collection of elements of $\mathbf{Q}$ that are used to label the hyperedges $e$ of $G$ such that the set of end vertices of $e$ is $V$. Moreover, the length of the collection is bounded by $\mu$ and is determined by $V$. By Lemmas \ref{lem_coverbqo} and \ref{lem_subseqbqo}, $\mathcal{G}$ is better-quasi-ordered by $\leqslant$.
\qed

In the following, we show that, for a better-quasi-ordered set $\mathbf{Q}$, the $\mathbf{Q}$-labeled hypergraphs of bounded (respectively, multiplicity) tree-width and tree-diameter are better-quasi-ordered by the (respectively, induced) subgraph relation.
\begin{lem}\label{lem_Qtwtdbqo}
Let $p, s \geqslant 0$ be integers, $\mathbf{Q}$ be a better-quasi-ordered set, $\mathcal{G}$ be the set of quintuples $\mathbf{G} := (G, T, \mathcal{V}, r, V_G)$, where $G$ is a $\mathbf{Q}$-labeled  hypergraph (respectively, of bounded multiplicity) with a tree-decomposition $(T, \mathcal{V})$ of width at most $p - 1$, $T$ is a rooted tree of root $r$ and height at most $s$, and $V_G \subseteq V_r$. Let $\lambda: V(G) \mapsto [p]$ be a colouring such that for each $v \in V(T)$, every pair of different vertices of $V_v$ are assigned different colours. For $\mathbf{X}, \mathbf{Y} \in \mathcal{G}$,
%(respectively, $\mathcal{G}^*$),
denote by $\mathbf{X} \subseteq \mathbf{Y}$ (respectively, $\mathbf{X} \leqslant \mathbf{Y}$) that there exists an isomorphism $\varphi$ from $X$ to a subgraph (respectively, an induced subgraph) of $\mathbf{Y}$ such that $\varphi(V_X) = V_Y$, and that for each $x \in V(X)$ and $e \in E(X)$, $\lambda(x) = \lambda(\varphi(x))$ and $\sigma(e) \leqslant_{\mathbf{Q}} \sigma(\varphi(e))$. Then $\mathcal{G}$ is better-quasi-ordered by $\subseteq$ (respectively, $\leqslant$).
\end{lem}
\pf
Let $\mathcal{G}_s$ be the set of $\mathbf{G} \in \mathcal{G}$ of which the height of $T$ is exactly $s$. By Lemma \ref{lem_coverbqo}, it is enough to prove the lemma for $\mathcal{G}_s$. The case of $s = 0$ is ensured by Lemma \ref{lem_nboundedbqo}. Inductively assume it holds for some $s - 1 \geqslant 0$. By Lemma \ref{lem_subseqbqo}, the powerset $\mathcal{M}_{s - 1}$ of $\mathcal{G}_{s - 1}$ is better-quasi-ordered.

Denote by $N_T(r)$ be the neighborhood of $r$ in $T$. For each $u \in N_T(r)$, let $T_u$ be the connected component of $T - r$ containing $u$, and $G_{T_u}$ be the subgraph of $G$ induced by the vertex set $\bigcup_{w \in V(T_u)}V_w$. Let $\mathcal{V}_{T_u} := \{V_w | w \in V(T_u)\}$, and $\mathbf{G}_{T_u} := (G_{T_u}, T_u, \mathcal{V}_{T_u}, u, V_r \cap V_u)$. Then $\mathbf{G}_{T_u} \in \mathcal{G}_{s - 1}$. Let $G_r$ be the subgraph of $G$ induced by $V_r$. Then $\mathbf{G}_r := (G_r, r, V_r, r, V_G) \in \mathcal{G}_0$. Clearly, $\mathbf{G} \mapsto \mathbf{G}_r \times \{\mathbf{G}_{T_u}| u \in N_T(r)\}$ is an order-preserving bijection from $\mathcal{G}_s$ to $\mathcal{G}_0 \times \mathcal{M}_{s - 1}$. By Lemma \ref{lem_coverbqo}, $\mathcal{G}_s$ is better-quasi-ordered since $\mathcal{G}_0$ and $\mathcal{M}_{s - 1}$ are better-quasi-ordered.
\qed

%Now we are ready to show a better-quasi-ordering result for graphs without $P_m$ as a subgraph.

We end this paper by proving that graphs (respectively, of bounded multiplicity) without a given path as a subgraph are better-quasi-ordered by the (respectively, induced) subgraph relation.

%\noindent {\bf Proof of Theorem \ref{thm_lpathfreebqo}.}
\begin{proof}[Proof of Theorem \ref{thm_lpathfreebqo}]
$(\Leftarrow)$ follows from Lemmas \ref{lem_tdibounded} and \ref{lem_Qtwtdbqo}.

$(\Rightarrow)$ Let $\mathcal{G}$ be the set of graphs without $H$ as a subgraph, quasi-ordered by the subgraph or induced subgraph relation. Suppose for a contradiction that $H$ is not a union of paths.
Then $H$ contains either a cycle or a vertex of degree at least $3$. For $i \geqslant 1$, let $C_i$ be the cycle of $|V(H)| + i$ vertices. Then $C_1, C_2, \ldots$ is a sequence without a good pair with respect to the subgraph or induced subgraph relation. So $\mathcal{G}$ is not well-quasi-ordered, not say better-quasi-ordered, a contradiction. Thus $H$ is a union of paths.
\end{proof}